\documentclass[reqno]{amsart}
\usepackage{amsfonts,amsmath,amsthm,amssymb,latexsym, cite}%
\usepackage{graphicx,verbatim, color}
\usepackage[colorlinks,plainpages,bookmarksnumbered]{hyperref}%, backref citecolor=red, linkcolor=blue,

\theoremstyle{plain}
\newtheorem{thm}{Theorem}[section]
\newtheorem{defn}[thm]{Definition}
\newtheorem{cor}[thm]{Corollary}
\newtheorem{lem}[thm]{Lemma}

\newtheorem{prop}[thm]{Proposition}

\numberwithin{equation}{section}

\newcommand{\dmn}{\mathop{\rm dom}}
\renewcommand{\Im}{\mathop{\rm Im}}
\newcommand{\supp}{\mathop{\rm supp}}

\renewcommand{\kappa}{\varkappa}

\newcommand{\Real}{\mathbb R}
\newcommand{\eps}{\varepsilon}
\newcommand{\en}{{\eps\nu}}
\newcommand{\el}{{\eps, \lambda\eps}}
\newcommand{\cI}{\mathcal{I}}
\newcommand{\cP}{\mathcal{P}}

\begin{document}

\title[1D Schr\"{o}dinger operators with short range interactions]
{1D Schr\"{o}dinger operators with short range interactions: two-scale regularization of distributional potentials}

\author{Yuriy Golovaty}%

\address{Department of Differential Equations,
  Ivan Franko National University of Lviv\\
  1 Universytetska str., 79000 Lviv, Ukraine}


\begin{abstract}
 For real $L_\infty(\Real)$-functions $\Phi$ and $\Psi$ of compact support, we prove the norm resolvent convergence, as $\eps$ and $\nu$ tend to $0$, of a family $S_\en$ of one-dimensional Schr\"odinger operators on the line of the form
 $$
    S_\en= -\frac{d^2}{dx^2}+\frac{\alpha}{\eps^2}\Phi\left(\frac{x}{\eps}\right)
    +\frac{\beta}{\nu}\Psi\left(\frac{x}{\nu}\right),
 $$
provided the ratio $\nu/\eps$ has a finite or infinite limit.
The limit operator $S_0$ depends on the shape of  $\Phi$ and $\Psi$ as well as on the limit of ratio $\nu/\eps$.
If the potential $\alpha\Phi$ possesses a zero-energy resonance,  then $S_0$ describes a non trivial  point interaction at the origin. Otherwise $S_0$ is the direct sum of the Dirichlet half-line Schr\"odinger operators.
\end{abstract}

\maketitle
%
%E-mail: \textit{yu\_\,holovaty@franko.lviv.ua}
%
%Affiliation:  \textit{Department of Differential Equations, Ivan Franko National University of L'viv}
%
%Address: \textit{Universytetska str. 1, L'viv 79000, Ukraine}

\section{Introduction}
The present paper is concerned with  convergence of  the family of one-dimensional Schr\"{o}dinger operators of the form
\begin{equation}\label{Sen}
    S_\en= -\frac{d^2}{dx^2}+\frac{\alpha}{\eps^2}\Phi\left(\frac{x}{\eps}\right)
    +\frac{\beta}{\nu}\Psi\left(\frac{x}{\nu}\right), \quad
\dmn S_\en=W_2^2(\mathbb{R})
\end{equation}
as the positive parameters $\nu$ and $\eps$ tend to zero simultaneously.
Here $\Phi$ and $\Psi$ are  real  potentials of compact supports, and $\alpha$ and $\beta$ are real coupling constants.

Our motivation of the study on this convergence comes from
an application to the scattering of quantum particles by $\delta$- and $\delta'$-shaped potentials, where
$\delta$ is the Dirac delta-function.
The potential in \eqref{Sen} is a two-scale regularization of the distribution $\alpha\delta'(x)+\beta \delta(x)$
provided that the conditions
\begin{equation}\label{PhiPsiConds}
   \int_{\Real}\Phi(t)\,dt=0, \qquad \int_{\Real}t\Phi(t)\,dt=-1\quad\text{and}\quad \int_{\Real}\Psi(t)\,dt=1
\end{equation}
hold.
Our purpose is to construct the so-called solvable mo\-dels describing with admissible fidelity the real quantum interactions governed by the Hamiltonian $S_\en$.
The quantum mechanical models that are based on the concept of point inter\-actions reveal an undoubted effectiveness whenever solvability together with non
triviality is required. It is an extensive subject with a large literature (see e.g. \cite{AlbeverioGesztesyHoeghKrohnHolden2edition, AlbeverioKurasov}, and the references given therein).

We  emphasize that all results presented here concern  arbitrary potentials $\Phi$ and $\Psi$ of compact support, and the $(\alpha\delta'+\beta \delta)$-like potentials satisfying conditions \eqref{PhiPsiConds}  are only a special case
in our considerations, the title of paper notwithstanding.
It is interesting to observe that if the first condition in \eqref{PhiPsiConds} is not fulfilled, then
these potentials do not converge even in the distributional sense. However,  surprisingly  enough,
the resolvents of  $S_\en$ still converge in norm.

We say that the Schr\"odinger operator~$-\frac{d^2}{d t^2}+\alpha\Phi$ in $L_2(\Real)$ possesses a \emph{half-bound state} (or \emph{zero-energy resonance}) if there exists a non trivial solution~$u_\alpha$ to the equation $-u'' +\alpha\Phi u= 0$
that is bounded on the whole line. The potential $\alpha\Phi$ is then called \emph{resonant}. In this case, we also say that $\alpha$ is a \textit{resonant coupling constant} for the potential $\Phi$.
Such a solution~$u_\alpha$ is  unique up to a scalar factor and has nonzero limits $u_\alpha(\pm\infty)=\lim_{x\to\pm\infty}u_\alpha(x)$ (see \cite{BolleGesztesyKlaus:1987, Klaus:1982}).
Our main result reads as follows.
\medskip

{\it
Let $\Phi$ and $\Psi$ be bounded real functions of compact support. Then the operator family $S_\en$ given by \eqref{Sen} converges as $\nu, \eps\to0$  in the norm resolvent sense, i.e., the resolvents $(S_\en-z)^{-1}$ converge in the uniform operator topology, provided the ratio $\nu/\eps$ has a finite or infinite limit.

\emph{Non-resonant case.}
If the potential $\alpha\Phi$ does not possess a zero-energy resonance, then the operators $S_\en$ converge
to the direct sum $S_-\oplus S_+$ of the Dirichlet half-line Schr\"odinger operators~$S_\pm$.

\emph{Resonant case.}
If the potential $\alpha\Phi$ is resonant with the half-bound state $u_\alpha$, then the limit operator $S$ is a perturbation of the free Schr\"odinger operator defined by $S\phi=-\phi''$
on functions $\phi$ in~$W_2^2(\Real\setminus\{0\})$, subject to the  boundary conditions at the origin
\begin{equation}\label{PointInteractionConds}
    \begin{pmatrix}
         \phi(+0)\\ \phi'(+0)
    \end{pmatrix} =
    \begin{pmatrix}
            \theta_\alpha(\Phi) & 0 \\
            \beta\,\omega_\alpha(\Phi,\Psi) & \theta_\alpha(\Phi)^{-1}
    \end{pmatrix}
    \begin{pmatrix}
      \phi(-0)\\ \phi'(-0)
    \end{pmatrix}.
\end{equation}
The diagonal matrix element $\theta_\alpha(\Phi)$ is specified by the half-bound state of potential $\alpha\Phi$, and is defined by
\begin{equation}\label{Theta}
    \theta_\alpha(\Phi)=\frac{u_\alpha^+}{u_\alpha^-},
\end{equation}
where $u_\alpha^\pm=u_\alpha(\pm\infty)$. The value $\omega_\alpha(\Phi,\Psi)$ depends on both potentials $\Phi$ and $\Psi$
as well as on the limit of ratio $\nu/\eps$ as $\nu, \eps\to 0$, and describes different kinds of the resonance interaction between the potentials $\Phi$ and $\Psi$.
 Three cases are to be distinguished:
\begin{itemize}
  \item[(i)] if $\nu/\eps\to\infty$ as $\nu, \eps\to 0$, then
\begin{equation}\label{OmegaZero}
    \omega_\alpha(\Phi,\Psi)=\frac{u_\alpha^+}{u_\alpha^-}\,\int_{\Real_+}\kern-4pt\Psi(t)\,dt
    +\frac{u_\alpha^-}{u_\alpha^+}\,\int_{\Real_-}\kern-4pt\Psi(t)\,dt;
\end{equation}
  \item[(ii)] if the ratio $\nu/\eps$ converges to a finite positive number $\lambda$ as $\nu, \eps\to 0$, then
  \begin{equation}\label{OmegaFinite}
    \omega_\alpha(\Phi,\Psi)=  \frac{1}{u_\alpha^-\,u_\alpha^+}\, \int_\Real\Psi(t)\, u^2_\alpha(\lambda t)\,dt;
\end{equation}
  \item[(iii)] if $\nu/\eps\to 0$ as $\nu$ and $\eps$ go to zero, then
\begin{equation}\label{OmegaInfty}
    \omega_\alpha(\Phi,\Psi)=  \frac{u^2_\alpha(0)}{u_\alpha^-\,u_\alpha^+} \,\int_\Real\Psi(t)\,dt.
\end{equation}
\end{itemize}
}

\medskip

The point interaction generated by  conditions \eqref{PointInteractionConds}
may be regarded as the first approximation to the real interaction governed by the Hamiltonian $S_\en$ with  coupling constants $\alpha$ lying in vicinity of the resonant values.
The explicit relations between the matrix entries $\theta_\alpha(\Phi)$, $\omega_\alpha(\Phi,\Psi)$ and the potentials $\Phi$, $\Psi$ make it possible to carry out a quantitative analysis of this quantum system, e.g. to compute approximate values of the scattering data. Of course the same conclusion holds in the non-resonant case, but then the quantum dynamics is asymptotically trivial.

It is natural to ask what happens if one of the coupling constants is zero, and the family $S_\en$ becomes one-parametric.
For if $\beta=0$, and so the $\delta$-like component of the short range potential is absent, then the results are in agreement with the results obtained recently in \cite{GolovatyHryniv:2010, GolovatyHryniv:2011}:
the operators
\begin{equation}\label{OprSe}
S_\eps= -\frac{d^2}{dx^2}+\frac{\alpha}{\eps^2}\Phi\left(\frac{x}{\eps}\right), \quad
\dmn S_\eps=W_2^2(\mathbb{R})
\end{equation}
converge as $\eps\to0$ in the norm resolvent sense  to the operator $S$ defined  by
conditions \eqref{PointInteractionConds} with $\beta=0$, if $\alpha\Phi$ possesses a zero-energy resonance, and to the direct sum~$S_-\oplus S_+$ otherwise.
As for the case $\alpha=0$, the limit Hamiltonian, as $\nu\to 0$, must be associated with  the $\beta\delta(x)$-interaction. However, we see at once that zero is a resonant coupling constant for any potential $\Phi$, and the half-bound state $u_0$ is a constant function. Therefore $\theta_0(\Phi)=1$, and $\omega_0(\Phi,\Psi)=\int_\Real\Psi\,dt$, no matter which a formula of \eqref{OmegaZero}--\eqref{OmegaInfty} we use.
Hence, the operator $S$ is defined by the boundary conditions
\begin{equation*}
\phi(+0)=\phi(-0),\qquad \phi'(+0)=\phi'(-0)+\beta\phi(0)\int_\Real\Psi\,dt,
\end{equation*}
as one should expect.

It has been believed for a  long time \cite{SebRMP:1986} that the  Hamiltonians $S_\eps$ given by \eqref{OprSe} with  $\alpha\neq 0$ converge as $\eps\to 0$ in the norm resolvent sense to the direct sum
$S_-\oplus S_+$ of the Dirichlet half-line Schr\"odinger operators for any potential $\Phi$ having zero mean.
If so, the $\delta'$-shaped potential defined through the regularization $\varepsilon^{-2}\Phi(\varepsilon^{-1}\,\cdot\,)$ must be opaque, i.e., acts as a perfect wall, in the limit $\eps\to 0$.
However, the numerical ana\-lysis of exactly solvable models of $S_\eps$ with piece-wise constant~$\Phi$ of compact support performed recently by Zolotaryuk a.o.~\cite{ChristianZolotarIermak03,Zolotaryuk08,Zolotaryuk09,Zolotaryuk10} gives rise to doubts that the limit $S_-\oplus S_+$ is correct.
The authors demonstrated that for a resonant~$\Phi$, the limiting value of the transmission coefficient of~$S_\eps$ is different from zero. The operators $S_\eps$
also arose in  \cite{AlbeverioCacciapuotiFinco:2007,CacciapuotiExner:2007, CacciapuotiFinco:2007} in connection with the approximation of  smooth planar quantum waveguides by  quantum graphs. Under the assumption that the mean value of $\Phi$ is different from zero, the authors singled out the set of resonant potentials~$\Phi$ producing a ``non-trivial'' (i.e., different from $S_-\oplus S_+$) limit of $S_\eps$ in the norm resolvent sense (see also the recent preprint \cite{Cacciapuoti:2011}). A similar resonance phenomenon  was also obtained in \cite{GolovatyManko:2009}, where the asymptotic behaviour of eigenvalues for the Schr\"{o}dinger operators perturbed by $\delta'$-like short range potentials was treated  (see also \cite{MankoJPA:2010}).
The situation with these controversial results was clarified in \cite{GolovatyHryniv:2010,GolovatyHryniv:2011}.
Note that \v{S}eba was the first \cite{SebaHalfLine:1985} who discovered  the ``resonant convergence'' for a similar family of the Dirichlet Schr\"{o}dinger operators
on the half-line.

There is a connection between the results presented here and the low energy behaviour  of Schr\"{o}dinger operators, in particular   the low-energy scattering theory.
Generally, the zero-energy resonances are the reason for different ``exceptional'' cases of the asymptotic behaviour.
Albeverio and H{\o}egh-Krohn~\cite{AlbeverioHoeghKrohn:1981} considered the family
of Hamiltonians $H_\eps=-\Delta + \lambda(\eps)\eps^{-2}V(\eps^{-1}x)$ in dimension three, where $\lambda(\eps)$ was a smooth function with $\lambda(0)=1$ and $\lambda'(0)\ne0$.
It was shown that $H_\eps$  converge  in the strong resolvent sense,  as $\eps\to0$, to the operator that is either the free Hamiltonian $-\Delta$ or its perturbation by a delta-function depending on whether or not there is a zero-energy resonance for $-\Delta+V$.
In~\cite{AlbeverioGesztesyHoeghKrohn:1982}, the low-energy scattering was discussed; the authors used the results of~\cite{AlbeverioHoeghKrohn:1981} and the connection between the low-energy behaviour of scattering matrix for the Hamiltonian $-\Delta + V$ in~$L_2(\Real^3)$ and for the corresponding scaled Hamiltonians
$-\Delta + \eps^{-2} V(\eps^{-1}x)$ as $\eps \to 0$ to study in detail possible resonant and non-resonant cases. Similar problem for Hamiltonians including the Coulomb-type interaction was treated in~\cite{AlbeverioGesztesyHoeghKrohnStreit:1983}.
The low-energy scattering for the one-dimensional Schr\"odinger operator $S_1$ and its connection to the behaviour of the corresponding scaled operators $S_\eps$ as $\eps\to0$ was thoroughly investigated by Boll\'e, Gesztesy, Klaus, and Wilk \cite{BolleGesztesyWilk:1985,BolleGesztesyKlaus:1987}, taking into  account the  possibility
of zero-energy resonances; in dimension two, the low-energy asymptotics was discussed in~\cite{BolleGesztesyDanneels:1988}. Continuity of the scattering matrix at zero energy for one-dimensional Schr\"odinger operators in the resonant case was established by Klaus in~\cite{Klaus:1988}. Relevant re\-ferences in this context are also \cite{AktosunKlaus:2001, DeiftTrubowitz:1979}.
Simon and Klaus  \cite{KlausSimonI:1980,KlausSimonII:1980, Klaus:1982} observed the connection between the  zero-energy resonances and the coupling  constant thresholds, i.e., the absorbtion  of eigenvalues.
These  results  depend  on  properties of  the corresponding  Birman-Schwinger  kernel.

Singular point interactions for the Schr\"odinger operators in dimensions one and higher have widely been discussed in both the physical and mathematical literature; see \cite{BrascheFigariTeta,ExnerNeidhardtZagrebnov,IsmagilovKostyuchenko:2010, Nizhik:2006FAA, BrascheNizhnik:2011, KostenkoMalamud:2010}.
It is worth to note that the considerable progress in theory of Schr\"{o}dinger operators with distributional potentials belonging to the Sobolev space $W_2^{-1}$ is due to Shkalikov, Savchuk \cite{ShkalikovSavchukMatNotes1999, ShkalikovSavchuk:2003TMMO}, and Mikhailets, Goriunov, and  Molyboga \cite{MikhailetsMolyboga:2008, MikhailetsMolyboga:2009, GoriunovMikhailetsMN:2010, GoriunovMikhailetsMFAT:2010}.

%%%%%%%%%%%%%%%%%%%%%%%%%%%%%%%%%%%%%%%%%%%%%%%%%%%%%%%%%%%%%%%%%%%%%%%%%%
% Preliminaries
%%%%%%%%%%%%%%%%%%%%%%%%%%%%%%%%%%%%%%%%%%%%%%%%%%%%%%%%%%%%%%%%%%%%%%%%%%

\section{Preliminaries}

There is no loss of generality in supposing that the supports of both $\Phi$ and $\Psi$ are contained  in the interval $\cI=[-1,1]$.
Denote by $\mathcal{P}$  the class of real-valued  bounded functions of compact support contained in $\cI$.

\begin{defn}
The \textit{resonant set} $\Lambda_\Phi$ of a potential $\Phi\in \mathcal{P}$ is the set of all real value $\alpha$ for which the operator $-\frac{d^2}{d t^2}+\alpha \Phi$ in $L_2(\Real)$ possesses a half-bound state, i.e., for which
there exists a non trivial $L_\infty(\Real)$-solution~$u_\alpha$ to the equation
\begin{equation}\label{EqUalpha}
- u'' +\alpha\Phi u= 0.
\end{equation}
\end{defn}

The half-bound state $u_\alpha$ is then constant outside the support of $\Phi$.
Moreover, the restriction of $u_\alpha$ to $\cI$ is a nontrivial solution of the Neumann boundary value problem
\begin{equation}\label{NeumanProblemWithAlpha}
     - u'' +\alpha \Phi u= 0, \quad t\in \cI,\qquad u'(-1)=0, \quad u'(1)=0.
\end{equation}
Consequently, for any $\Phi\in \cP$  the resonant set $\Lambda_\Phi$ is not empty and coincides with the set of all eigenvalues of  the latter problem with respect to the spectral parameter $\alpha$.
In the case of a nonnegative (resp. nonpositive) potential $\Phi$ the spectrum of  \eqref{NeumanProblemWithAlpha}
is discrete and simple with one accumulation point at $-\infty$ (resp. $+\infty$).
Otherwise, \eqref{NeumanProblemWithAlpha} is a  problem with  indefinite  weight function $\Phi$, and has a discrete and simple spectrum with two accumulation points at $\pm\infty$ \cite{CurgusLangerJDE:1989}.

We introduce some characteristics of the potentials $\Phi$ and $\Psi$.
Let $\theta$  be the map of $\Lambda_\Phi$ to $\Real$ defined by
\begin{equation*}%\label{MapTheta}
\theta(\alpha)=\frac{u_\alpha^+}{u_\alpha^-}=\frac{u_\alpha(+1)}{u_\alpha(-1)}.
\end{equation*}
Since the half-bound state is  unique up to a scalar factor, this map is well defined.
Throughout the paper, we choose the half-bound state so that $u_\alpha(x)=1$ for  $x\leq-1$.
Then $\theta(\alpha)=u_\alpha^+$, and  $u_\alpha(x)=\theta(\alpha)$ for  $x\geq1$.
Here and subsequently, $\theta_\alpha$ stands for the value $\theta(\alpha)$.
For our purposes it is convenient to introduce the maps:
\begin{align}
\label{MapZeta}
    &\zeta\colon \Lambda_\Phi\to \Real,
    &&\zeta(\alpha)=\theta_\alpha\int_{\Real_+}\kern-4pt \Psi\,dt+
    \theta_\alpha^{-1}\int_{\Real_-}\kern-4pt \Psi\,dt;\\
\label{MapKappa}
    &\kappa\colon \Lambda_\Phi\times \Real_+\to \Real,
    &&\kappa(\alpha,\lambda)= \theta_\alpha^{-1} \int_\Real\Psi(t)\, u^2_\alpha(\lambda t)\,dt;\\
\label{MapMu}
    &\mu\colon \Lambda_\Phi\to \Real,
    &&\mu(\alpha)=  \theta_\alpha^{-1}u^2_\alpha(0) \int_\Real\Psi\,dt
\end{align}
(compare with \eqref{OmegaZero}--\eqref{OmegaInfty}).

Denote by $S(\gamma_1,\gamma_2)$ a perturbation of the free Schr\"odinger operator acting via $S(\gamma_1,\gamma_2)\phi=-\phi''$
on functions $\phi$ in~$W_2^2(\Real\setminus\{0\})$ obeying the  interface conditions
$\phi(+0) = \gamma_1 \phi(-0)$ and $\phi'(+0) = \gamma_1^{-1} \phi'(-0)+\gamma_2 \phi(-0)$ at the origin.
For every real $\gamma_1$ and $\gamma_2$, this operator is  self-adjoint
provided $\gamma_1\neq 0$.
Let $S_\pm$  denote  the unperturbed half-line Schr\"odinger operator
$S_\pm= -d^2/d x^2$ on~$\Real_\pm$, subject to the Dirichlet boundary condition at the origin, i.e.,
$$
    \dmn S_\pm = \{ \phi \in W_2^2(\Real_\pm) \colon \phi(0)=0\}.
$$

In the sequel, letters $C_j$ and $c_j$  denote various posi\-ti\-ve constants independent of~$\eps$ and $\nu$, whose values might be different in different proofs. Throughout the paper, $W_2^l(\Omega)$ stands for the Sobolev space and $\|f\|$ stands for the $L_2(\Real)$-norm of a function~$f$.

We start with  an easy auxiliary result, which  will be often used below.
\begin{prop}\label{PropEstYoverF}
Assume $f\in L_2(\Real)$, $z\in \mathbb{C}\setminus\Real$, and set $y=(S(\gamma_1,\gamma_2)-z)^{-1}f$.
Then the following holds  for some constants $C_k$ independent of $f$ and  $t$:
\begin{align}\label{EstY(pm0)}
    &|y(\pm 0)|\leq C_1\|f\|, && |y'(\pm 0)|\leq C_2\|f\|\\\label{EstY(t)-Y(0)}
&\bigr|y(\pm t)-y(\pm 0)\bigl|\leq C_3t\|f\|, &&
\bigr|y'(\pm t)-y'(\pm 0)\bigl|\leq C_4 t^{1/2}\|f\|
\end{align}
for $t>0$. These inequalities  hold also for $y=(S_-\oplus S_+-z)^{-1}f$.
\end{prop}
\begin{proof}
We first observe that $(S(\gamma_1,\gamma_2)-z)^{-1}$ is a bounded operator from~$L_2(\Real)$ to the domain of~$S(\gamma_1,\gamma_2)$ equipped with the graph norm. The latter space is continuously embedded  subspace into $W_2^2(\Real\setminus \{0\})$. Then $ \|y\|_{W_2^2(\Real\setminus \{0\})}\leq c_1\|f\|$. Owing to the Sobolev embedding theorem, we have
 $\|y\|_{C^1(\Real\setminus \{0\})}\leq c_2\|f\|$,
which  establishes  \eqref{EstY(pm0)}. Combining the previous estimates for $y$ with the inequalities
\begin{equation*}
\bigr|y^{(j)}(\pm t)-y^{(j)}(\pm 0)\bigl|\leq \left|\int_0^{\pm t}|y^{(j+1)}(s)|\,ds\right|,\quad j=0,1,
\end{equation*}
we obtain \eqref{EstY(t)-Y(0)}. For the case of $S_-\oplus S_+$, the proof is similar.
\end{proof}

Apparently, some versions of  the next proposition are known, but  we  are at a loss to give a precise reference.

\begin{prop}\label{PropCauhyProblEst}
Let $J$ be a finite interval in $\Real$, and $t_0\in J$.
Then the solution  to the Cauchy problem
    $v''+qv=f$ in $J$, $v(t_0)=a$, $v'(t_0)=b$
obeys the estimate
$$
    \|v\|_{C^1(J)}\leq C(|a|+|b|+\|f\|_{L_\infty(J)})
$$
for some $C>0$ being independent of the initial data and right-hand side, whenever $q, f\in L_\infty(J)$.
\end{prop}
\begin{proof}
Let $v_1$ and $v_2$ be the linear independent solutions to $v''+qv=0$ such that
$v_1(t_0)=1$, $v'_1(t_0)=0$, $v_2(t_0)=0$ and $v'_2(t_0)=1$. Under the assumptions made on $q$ and $f$, these solutions belong to $W_2^2(J)$; and consequently $v_j\in C^1(J)$ by the Sobolev embedding theorem.  Application of the variation of parameters method  yields
     \begin{equation}\label{CPSolRepresentation}
        v(t)=a v_1(t)+b v_2(t)+\int_{t_0}^tk(t,s)f(s)\,ds,
     \end{equation}
where $k(t,s)=v_1(s)v_2(t)-v_1(t)v_2(s)$.
From this and the representation of the first derivative
     \begin{equation*}
        v'(t)=a v'_1(t)+b v'_2(t)+\int_{t_0}^t \frac{\partial k}{\partial t}(t,s)f(s)\,ds
     \end{equation*}
we have
\begin{equation*}
    |v(t)|+|v'(t)|\leq |a|\|v_1\|_{C^1(J)}+|b|\|v_2\|_{C^1(J)}+2|J|\,\|k\|_{C^1(J\times J)}\|f\|_{L_\infty(J)}
\end{equation*}
for $t\in J$, which completes the proof.
\end{proof}

We end this section with a proposition which will be useful in Sections~\ref{SecZero} and~\ref{SecInfty}.
Denote by $[\,\cdot\,]_b$  the jump of a function at the point $x=b$.
\begin{prop}\label{PropW22Corrector}
Let $\Real_a$ be the real line with two removed points $-a$ and $a$, i.e., $\Real_a=\Real\setminus \{-a,a\}$.
    Assume  $w\in W_2^2(\Real_a)$. There exists a function $r\in C^\infty(\Real_a)$ such that   $w+r$ belongs to $W_2^2(\Real)$, $r$ is zero in $(-a,a)$, and
    \begin{equation}\label{REst}
        \max_{x\in \Real_a}|r^{(k)}(x)|\leq C \Bigl(\left|[w]_{-a}\right|+\left|[w]_{a}\right|+\left|[w']_{-a}\right|+\left|[w']_{a}\right|\Bigr)
    \end{equation}
    for $k=0,1,2$, where the constant $C$ does not depend on $w$ and $a$.
\end{prop}
\begin{proof}
Let us introduce functions $\varphi$ and $\psi$ that are smooth outside the origin, have compact supports contained in $[0,\infty)$, and $\varphi(+0)=1$, $\varphi'(+0)=0$, $\psi(+0)=0$, $\psi'(+0)=1$.
Set
\begin{equation}\label{CorrectoR}
r(x)=[w]_{-a}\, \varphi(-x-a)-[w']_{-a}\,\psi(-x-a)-[w]_{a}\,\varphi(x-a)-[w']_{a}\,\psi(x-a).
\end{equation}
All jumps are well defined, since $w\in C^1(\Real_a)$.
Next, the function $r$ is zero in $(-a,a)$ by construction.
An easy computation shows that  $w+r$
is continuous on~$\Real$ along with its derivative and consequently belongs to $W_2^2(\Real)$. Finally, \eqref{CorrectoR} makes it obvious that inequality \eqref{REst} holds.
\end{proof}

%%%%%%%%%%%%%%%%%%%%%%%%%%%%%%%%%%%%%%%%%%%%%%%%%%%%%%%%%%%%%%%%%%%%%%%%%%%%%%%
% Case \eps\nu^{-1}\to 0
%%%%%%%%%%%%%%%%%%%%%%%%%%%%%%%%%%%%%%%%%%%%%%%%%%%%%%%%%%%%%%%%%%%%%%%%%%%%%%%

\section{Convergence of the operators $S_\en$. The case  $\nu\eps^{-1}\to \infty$.}\label{SecZero}

In this section, we analyze the case of a ``$\delta$-like'' sequence that is slowly contracting  relative to ``$\delta'$-like'' one. The relations between two parameters $\eps$ and $\nu$ that lead to this case are, roughly speaking, as follows: $\eps\ll 1$, $\nu \ll 1$, but $\nu/\eps\gg 1$.  It will be convenient to introduce the large parameter~$\eta=\nu/\eps$.
The  first trivial observation is the following: if $\nu\to 0$ and $\eta\to \infty$, then $\eps\to 0$.
The resonant and non-resonant cases will be considered separately.

\subsection{Resonant case}
We start with the analysis of the more difficult resonant case.
Suppose that $\alpha\in \Lambda_\Phi$ and set $\zeta_\alpha=\zeta(\alpha)$, where $\zeta$ is given by \eqref{MapZeta}.

\begin{thm}\label{ThmCaseEpsNu-1Go0}
    Assume  $\Phi, \Psi \in \mathcal{P}$ and $\alpha$ belongs to the resonant set $\Lambda_\Phi$. Then the operator family $S_\en$ defined by \eqref{Sen} converges to  the operator $S(\theta_\alpha, \beta\zeta_\alpha)$
    as $\nu \to 0$ and $\eta\to \infty$  in the norm resolvent sense.
\end{thm}
\medskip

We have divided the proof into a sequence of lemmas.

Let us fix a function $f\in L_2(\Real)$ and a number $z\in \mathbb{C}$ with $\Im z\neq 0$. For abbreviation, in this section we let $S$ stand for $S(\theta_\alpha, \beta\zeta_\alpha)$.
Our aim is to approximate both vectors $(S_\en-z)^{-1}f$ and $(S-z)^{-1}f$ in $L_2(\Real)$  by the \textit{same} element $y_\en$ from the domain of $S_\en$.
Of course, such an approximation must be uniform in $f$ in bounded subsets of $L_2(\Real)$. We construct the vector $y_\en$ in the explicit form, which allows us to estimate $L_2(\Real)$-norms of the differences
$(S_\en-z)^{-1}f-y_\en$ and $(S-z)^{-1}f-y_\en$.
This is the aim of the next lemmas.

First we construct a candidate for the approximation as follows. Let us set $y=(S-z)^{-1}f$. Write
$w_\en(x)=y(x)$ for $|x|>\nu$ and
\begin{equation*}
  w_\en(x)=y(-0)\bigl(u_\alpha(x/\eps)+\beta\nu h_\en(x/\nu)\bigr)+\eps g_\en( x/\eps)
  +\eps^2 v_\en(x/\eps)\quad \text{for }|x|\leq\nu.
\end{equation*}
Here $h_\en$, $g_\en$, and $v_\en$  are  solutions to the Cauchy problems
\begin{align}
\label{ProblHen}
     &\hskip12pt h''= \Psi(t)u_\alpha\left(\eta t\right),\quad t\in\Real,\qquad
      h(0)=0,\quad h'(0)=0;
\\ \label{ProblGen}
    &\begin{cases}\displaystyle
     g''-\alpha \Phi(t)g= \alpha \beta \eta y(-0)\Phi(t)h_\en\left(\eta t\right),\quad t\in\Real,\\\displaystyle
      g(-1)=0, \quad g'(-1)=y'(-0)+\beta y(-0)\int_{\Real_-}\kern-4pt \Psi\,ds;
    \end{cases}
 \\ \label{ProblVen}
&\hskip8pt -v''+\alpha\Phi(t)v=f(\eps t)\chi_\eta(t),\quad t\in\Real,\quad v(0)=0,\; v'(0)=0
\end{align}
respectively, and $u_\alpha$ is the half-bound state corresponding to the resonant coupling constant $\alpha$. Here and subsequently, $\chi_a$  is the characteristic function of interval $(-a,a)$.
Hence we can surely expect that $y$  is a very satisfactory approximation to $(S_\en-z)^{-1}f$  for $|x|>\nu$, but
the approximation on the support of $\Psi$ is more subtle.

\begin{lem}\label{LemmaHenProperties}
The function $h_\en$ possesses the following properties:

(i) there exist constants $C_1$ and $C_2$ such that
\begin{equation}\label{HenEst}
  \|h_\en\|_{C^1(\cI)}\leq C_1,\qquad
  |h_\en(t)|\leq C_2\, t^2
\end{equation}
for $t\in\Real$ and all  $\eps, \nu \in (0,1)$;

(ii) the asymptotic relations
\begin{equation}\label{HenPrimeAtPm1}
    h_\en'(-1)=- \int_{\Real_-}\kern-4pt \Psi\,ds+O(\eta^{-1}),\qquad
    h_\en'(1)= \theta_\alpha\int_{\Real_+}\kern-4pt \Psi\,ds+O(\eta^{-1})
\end{equation}
hold as $\nu\to0$ and $\eta\to \infty$.
\end{lem}

\begin{proof}
The solution $h_\en$ and its derivative can be represented as
\begin{equation}\label{Hen}
    h_\en(t)= \int_0^t(t-s)\Psi(s)u_\alpha (\eta s)\,ds, \qquad
    h'_\en(t)= \int_0^t\Psi(s)u_\alpha (\eta s)\,ds.
\end{equation}
The first estimate in \eqref{HenEst} follows immediately from these relations, because
$\Psi$ and $u_\alpha$ belong to $L_\infty(\Real)$.  By the same reason,
 \begin{equation*}
 |h_\en(t)|\leq c_1\left|\int_0^t|t-s|\,ds\right|\leq C_2 t^2.
 \end{equation*}
Now according to our choice of the half-bound state, we see that
 \begin{equation*}
    u_\alpha (\eta t)\to
    u_\alpha^*(t)= \begin{cases}
        1 & \text{if } t<0,\\
        \theta_\alpha & \text{if } t>0
    \end{cases}
\end{equation*}
in $L_{1, loc}(\Real)$, as $\eta\to \infty$. In addition,  the difference $u_\alpha(\eta t)-u_\alpha^*(t)$ is zero outside the interval $[-\eta^{-1},\eta^{-1}]$ and bounded on this interval. In view of the second relation in \eqref{Hen},  this establishes the asymptotic formulas \eqref{HenPrimeAtPm1}.
\end{proof}

\begin{lem}\label{LemmaGenProperties}
There exist constants $C_1$ and $C_2$, independent of $f$, such that
\begin{align}\label{GenEstT}
    &|g_\en(t)|\leq C_1(1+|t|)\|f\|, &&t\in\Real,\\\label{GenPrimeEstR}
    &|g'_\en(t)|\leq C_2\|f\|,  &&t\in\Real
\end{align}
for all  $\eps$ and $\nu$ whenever the ratio of $\eps$ to $\nu$ remains bounded as $\eps, \nu\to 0$.
In addition,  the value $g'_\en(1)$  admits the asymptotics
  \begin{equation}\label{GenPrimeAt1}
     g'_\en(1)=\theta_\alpha^{-1}\left(y'(-0)+\beta y(-0)\int_{\Real_-}\kern-4pt \Psi\,ds\right)+O(\eta^{-1})\|f\|
  \end{equation}
as  $\nu\to 0$,  $\eta\to \infty$.
\end{lem}
\begin{proof}
 From Proposition~\ref{PropCauhyProblEst} it follows that
\begin{equation*}%\label{GenC1Est}
    \|g_\en\|_{C^1(\cI)}\leq c_1(|y(-0)|+|y'(-0)|)+
    c_2 \eta |y(-0)|\,\|h_\en(\eta^{-1}\,\cdot\,)\|_{C(\cI)}.
\end{equation*}
Next, in light of \eqref{HenEst}, we have
\begin{equation}\label{HenEstEn}
    \|h_\en(\eta^{-1}\,\cdot\,)\|_{C(\cI)}=
    \max_{\phantom{1}|t|\leq\eta^{-1}}|h_\en(t)|\leq c_3\eta^{-2}.
\end{equation}
Combining this estimate with \eqref{EstY(pm0)}, we deduce
\begin{equation}\label{GenEstC1(-1,1)}
\|g_\en\|_{C^1(\cI)}\leq c_4(|y(-0)|+|y'(-0)|)\leq c_5\|f\|.
\end{equation}
Since the support of $\Phi$ lies in $\cI$, the function $g_\en$ is linear outside $\cI$, namely
$g_\en(t)=g_\en'(-1)(t+1)$ for $t\leq-1$ and $g_\en(t)=g_\en(1)+g_\en'(1)(t-1)$ for $t\geq1$.
Therefore estimates \eqref{GenEstT}, \eqref{GenPrimeEstR} follow easily from these relations and \eqref{GenEstC1(-1,1)}.

Next,  multiplying equation \eqref{ProblGen} by $u_\alpha$ and integrating on $\cI$ by parts  yield
\begin{equation*}%\label{CondBoundedSolution}
   \theta_\alpha g'_\en(1)- g'_\en(-1)=
 \alpha \beta\eta\, y(-0)\int_{-1}^1\Phi(s)\,h_\en\left(\eta^{-1} s\right)u_\alpha(s)\,ds.
\end{equation*}
The right-hand side can be estimated by $c_6\eta^{-1}\|f\|$ provided $|\eta|\geq1$, in view of \eqref{HenEstEn} and Proposition~\ref{PropEstYoverF}.  Recalling the initial conditions \eqref{ProblGen}, we obtain \eqref{GenPrimeAt1}.
\end{proof}

\begin{lem}\label{LemmaVenProperties}
There exist constants $C_1$ and $C_2$, independent of $f$, such that
\begin{equation}\label{VenEst}
   |v_\en(t)|\leq C_1\eps^{-2}\nu^{3/2}\|f\|,\qquad |v'_\en(t)|\leq C_2\eps^{-1}\nu^{1/2}\|f\|
\end{equation}
for $t\in [-\eta,\eta]$, as $\nu\to0$ and $\eta\to \infty$.
\end{lem}
\begin{proof}
  The proof consists in the careful analysis of representation \eqref{CPSolRepresentation} for the case of problem  \eqref{ProblVen}. In fact,
  \begin{equation*}
    v_\en(t)=\int_{0}^tk(t,s)f(\eps s)\chi_\eta(s)\,ds,
  \end{equation*}
where $k(t,s)=v_1(s)v_2(t)-v_1(t)v_2(s)$, and $v_1$, $v_2$ are solutions of $-v''+\alpha\Phi v=0$ subject to the initial conditions $v_1(0)=1$, $v'_1(0)=0$ and $v_2(0)=0$, $v'_2(0)=1$ respectively.

The kernel $k$ admits the following estimates
\begin{equation}\label{Kest}
  |k(t,s)|\leq c_1 (|t|+|s|)+c_2, \quad \left|\frac{\partial k}{\partial t}(t,s)\right|\leq c_3,
  \quad (t,s)\in\Real^2
\end{equation}
with some positive constants $c_j$.
Indeed, both solutions $v_1$ and $v_2$ are linear functions outside the interval $\cI$, since $\supp \Phi\subset \cI$. Set $v_j(t)=a_j^\pm t+b_j^\pm$ for $\pm t>1$.
Suppose that $t>1$ and $s>1$;
then
\begin{equation*}
  k(t,s)=(b_1^+a_2^+-b_2^+a_1^+)(t-s),  \quad \frac{\partial k}{\partial t}(t,s)=b_1^+a_2^+-b_2^+a_1^+,
\end{equation*}
which implies \eqref{Kest}  for such $t$ and $s$. Next, if $t>1$ and $|s|<1$, then
\begin{equation*}
  k(t,s)=v_1(s)(a_2^+t+b_2^+)-v_2(s)(a_1^+t+b_1^+),  \quad \frac{\partial k}{\partial t}(t,s)=a_2^+v_1(s)-a_1^+v_2(s).
\end{equation*}
That \eqref{Kest} for such $t$ and $s$ follows from the estimates $\|v_j\|_{C(-1,1)}\leq c_4$, $j=1,2$.
The other cases (such as $|t|<1$ and $s>1$; $t<-1$ and $s<-1$, and  so on) can be treated in a similar way.

Therefore, for $\eta$ large enough, we have
\begin{align*}
&\begin{aligned}
    \max_{t\in [-\eta,\eta]}|v_\en(t)|\leq
    \int_{-\eta}^\eta \max_{t\in [-\eta,\eta]}|k(t,s)||f(\eps s)|\,ds
   \leq
    \int_{-\eta}^\eta (c_5 (\eta+|s|)+c_6)|f(\eps s)|\,ds
    \\
    \leq
     c_7 \eta\int_{-\eta}^\eta |f(\eps s)|\,ds = c_7 \eta\eps^{-1}\int_{-\nu}^\nu |f(\tau)|\,d\tau
    \leq c_8 \eta\eps^{-1}\nu^{1/2}\|f\|=c_8 \eta\eps^{-2}\nu^{3/2}\|f\|,
\end{aligned}\\
&\begin{aligned}
  \max_{t\in [-\eta,\eta]}|v'_\en(t)|&\leq
    \int_{-\eta}^\eta \max_{t\in [-\eta,\eta]}\left|\frac{\partial k}{\partial t}(t,s)\right| |f(\eps s)|\,ds\\
    &\leq
    c_9\int_{-\eta}^\eta|f(\eps s)|\,ds
   \leq c_{10} \eps^{-1}\int_{-\nu}^\nu |f(\tau)|\,d\tau
    \leq c_{11}\eps^{-1}\nu^{1/2}\|f\|,
\end{aligned}
\end{align*}
which proves the lemma.
\end{proof}

\begin{cor}\label{CorWenSmall}
The function $w_\en$ is bounded in $[-\nu,\nu]$ uniformly in $\eps$ and $\nu$ provided the ratio $\eps/\nu$ remains bounded as $\eps, \nu \to 0$, and there exists a constant $C$ such that
 $\max_{|x|\leq \nu} |w_\en(x)|\leq C\|f\|$.
\end{cor}
\begin{proof}
    The corollary is a direct consequence of Lemmas~\ref{LemmaHenProperties}--\ref{LemmaVenProperties}. We only note that
    \begin{multline}\label{EstEpsG}
    \max_{|x|\leq \nu}|\eps g_\en( x/\eps)+\eps^2 v_\en(x/\eps)|\leq \bigl(c_1\eps(1+\nu/\eps)+c_2\nu^{3/2}\bigr)\|f\|\\
    \leq c_3(\eps+\nu)\|f\|\leq c_4\nu\|f\|,
    \end{multline}
in view of \eqref{GenEstT}, \eqref{VenEst}, and the assumption that $\eps\leq c\nu$.
\end{proof}

By construction,  $w_\en$  belongs to $W_2^2(\Real\setminus\{-\nu,\nu\})$.
In general, due to the discontinuity  at the points $x=\pm\nu$,  $w_\en$  is not an element of $\dmn S_\en$.
However,  the jumps of $w_\en$ and the jumps of its first derivative at these points are small enough, as shown below. By Proposition~\ref{PropW22Corrector}, there exists the corrector function $r_\en$ of the form \eqref{CorrectoR}
such that $w_\en+r_\en$  belongs to $W_2^2(\Real) = \dmn S_\en$. Set $y_\en=w_\en+r_\en$.

\begin{lem}\label{LemmaJumpsAreSmall}
The corrector $r_\en$ is small as $\nu\to 0$, $\eta\to \infty$,   and satisfies the inequality
\begin{equation*}%\label{RenEst}
    \max_{x\in \Real\setminus\{-\nu,\nu\}}\bigl|r^{(k)}_\en(x)\bigr|\leq C\varrho(\nu,\eta)\|f\|
\end{equation*}
for $k=0,1,2$, where $\varrho(\nu,\eta)=\nu^{1/2}+\eta^{-1}$.
\end{lem}

\begin{proof}
Assume  $\eps$ and $\nu$ are  small enough, and  $\eta\geq 1$. From our choice of $u_\alpha$, we have that
$u_\alpha(-\eta)=1$, $u_\alpha(\eta)=\theta_\alpha$, and $u'_\alpha(\pm\eta)=0$. Also $g_\en'(\pm \eta)=g_\en'(\pm1)$, and the bounds
\begin{equation}\label{EstEpsG(eps/nu)}
    \eps|g_\en(\pm\eta)|\leq c_1 \nu\|f\|
\end{equation}
hold, owing to \eqref{EstEpsG}.
These relations will be used repeatedly in the proof.

According to Proposition~\ref{PropW22Corrector}, it is sufficient to estimate the jumps of $w_\en$ and $w'_\en$. At the point $x=-\nu$ we have
\begin{align*}
    [w_\en]_{-\nu}&=y(-0)+\beta\nu y(-0) h_\en(-1)+\eps g_\en(-\eta)+\eps^2v_\en(-\eta)-y(-\nu),\\
    [w'_\en]_{-\nu}&=\beta y(-0) h'_\en(-1)+ g'_\en(-1)+\eps v'_\en(-\eta) -y'(-\nu).
\end{align*}
 The first of these jumps can be bounded as follows:
\begin{multline*}
    |[w_\en]_{-\nu}|\leq |y(-0)-y(-\nu)|+\nu|\beta| |y(-0)| |h_\en(-1)|\\
    +\eps |g_\en(-\eta)|+\eps^2|v_\en(-\eta)|\leq c_2\nu\|f\|,
\end{multline*}
by \eqref{HenEst}, \eqref{EstEpsG(eps/nu)}, Proposition \ref{PropEstYoverF}, and Lemma~\ref{LemmaVenProperties}.
Next, taking into account \eqref{HenPrimeAtPm1} and the initial conditions for $g_\en$, we see that
\begin{align*}
    [w'_\en]_{-\nu}&=\beta y(-0) \Bigl(- \int_{\Real_-}\kern-4pt \Psi\,ds+O(\eta^{-1})\Bigr)+ y'(-0)+\beta y(-0)\int_{\Real_-}\kern-4pt \Psi\,ds-y'(-\nu)\\&+\eps v'_\en(-\eta)
= y'(-0)-y'(-\nu)+O(\eta^{-1})y(-0)+O(\nu^{1/2})\|f\|,
\end{align*}
as $\eta\to \infty$ and $\nu\to 0$.
We can now repeatedly apply Proposition \ref{PropEstYoverF} to deduce $\left|[w'_\en]_{-\nu}\right|\leq c_3\varrho(\nu,\eta)\|f\|$.

Let us turn  to the jumps  at the point $x=\nu$. We get
\begin{align*}
    [w_\en]_{\nu}&=y(\nu)-\theta_\alpha y(-0)-\beta\nu y(-0) h_\en(1)-\eps g_\en(\eta)-\eps^2v_\en(\eta),\\
    [w'_\en]_{\nu}&=y'(\nu)-\beta y(-0) h'_\en(1)-g'_\en(1)-\eps v'_\en(\eta).
\end{align*}
Recall that  $y(+0)=\theta_\alpha y(-0)$, since $y\in \dmn S$.
This gives
$$
\left|[w_\en]_{\nu}\right|\leq |y(\nu)-y(+0)|+c_4 \nu |y(-0)| +\eps|g_\en(\eta)|+\eps^2|v_\en(\eta)|\leq c_5\nu\|f\|
$$
by \eqref{EstY(t)-Y(0)}, \eqref{VenEst}, and \eqref{EstEpsG(eps/nu)}.
Also, combining the  relation $y'(+0)=\theta_\alpha^{-1} y'(-0)+\beta\zeta_\alpha y(-0)$ and asymptotic formulas \eqref{HenPrimeAtPm1}, \eqref{GenPrimeAt1}, we deduce that
\begin{align*}
    [w'_\en]_{\nu}&=y'(\nu)-\beta y(-0) \Bigl(\theta_\alpha\int_{\Real_+}\kern-4pt \Psi\,ds+O(\eta^{-1})\Bigr)
\\
    &\phantom{=y'(\nu)\,}-\Bigl( \theta_\alpha^{-1}y'(-0)+\theta_\alpha^{-1}\beta y(-0)\int_{\Real_-}\kern-4pt \Psi\,ds+O(\eta^{-1})\|f\|\Bigr)-\eps v'_\en(\eta)
\\
    &=y'(\nu)-\theta_\alpha^{-1}y'(-0)-\beta\zeta_\alpha y(-0) +O(\eta^{-1})\|f\|+O(\nu^{1/2})\|f\|
\\
    &=y'(\nu)-y'(+0)+O(\eta^{-1}+\nu^{1/2})\|f\|,
\end{align*}
hence that $\left|[w'_\en]_{\nu}\right|\leq c_6\varrho(\nu,\eta)\|f\|$.
This inequality completes the proof.
\end{proof}

\begin{proof}[Proof of Theorem~\ref{ThmCaseEpsNu-1Go0}]
We first compute $(S_\en-z)y_\en$.
For the convenience of the reader we write $y_\en=w_\en+r_\en$ in the detailed form
\begin{equation}\label{hatYen}
    y_\en(x)=
    \begin{cases}
        y(x)+r_\en(x) &\text{if }|x|>\nu,\\
        y(-0)\bigl(u_\alpha(x/\eps)+\nu \beta h_\en(x/\nu)\bigr)+\eps g_\en( x/\eps)+\eps^2v_\en(x/\eps) &\text{if }|x|\leq\nu.
    \end{cases}
\end{equation}
Recall that $r_\en$ is zero in $(-\nu,\nu)$, by construction. Set $f_\en=(S_\en-z)y_\en$.
If $|x|>\nu$, then
\begin{equation*}
    f_\en(x)=
   \left(-\tfrac{d^2}{dx^2}-z\right)y_\en(x)=f(x)-r''_\en(x)-z r_\en(x).
\end{equation*}
Next, for $|x|<\nu$, we have
\begin{align*}\textstyle
  f_\en(x)=&\left(-\tfrac{d^2}{dx^2}+\alpha\eps^{-2}\Phi\left(\tfrac{x}{\eps}\right)
    +\beta\nu^{-1}\Psi\left(\tfrac{x}{\nu}\right)-z\right)y_\en(x)
\\
        = &\eps^{-2}\, y(-0)\Bigl\{ -u''_\alpha\left(\tfrac x\eps\right)+\alpha \Phi\left(\tfrac x\eps\right)u_\alpha\left(\tfrac x\eps\right)\Bigr\}
\\
     +&\nu^{-1} \,  \beta y(-0) \Bigl\{ -h''_\en\left(\tfrac x\nu\right)+\Psi\left(\tfrac x\nu\right)u_\alpha\left(\tfrac{x}{\eps}\right)\Bigr\}
\\
      +&\eps^{-1}\,
      \Bigl\{ -g_\en''\left(\tfrac x\eps\right)+\alpha \Phi\left(\tfrac x\eps\right)g_\en\left(\tfrac x\eps\right)
      +\eta\alpha \beta y(-0)\Phi\left(\tfrac x\eps\right)h_\en\left(\tfrac x\nu\right)\Bigr\}
\\
       +&\Bigl\{-v''_\en\left(\tfrac x\eps\right)+\alpha\Phi\left(\tfrac x\eps\right)v_\en\left(\tfrac x\eps\right)\Bigr\}
\\
     +&\beta\Psi\left(\tfrac x\nu\right)\Bigl\{\beta y(-0) h_\en\left(\tfrac x\nu\right)+\eta^{-1} g_\en\left(\tfrac x\eps\right)+\eps\eta^{-1} v_\en\left(\tfrac x\eps\right)\Bigr\}
     - z y_\en(x)
\\
    =&f(x)+\beta\Psi\left(\tfrac x\nu\right)\Bigl\{\beta y(-0) h_\en\left(\tfrac x\nu\right)+\eta^{-1} g_\en\left(\tfrac x\eps\right)+\eps\eta^{-1} v_\en\left(\tfrac x\eps\right)\Bigr\}- z y_\en(x),
\end{align*}
since  $u_\alpha$, $h_\en$, $g_\en$, and $v_\en$ are solutions to equations \eqref{EqUalpha},
\eqref{ProblHen}--\eqref{ProblVen}  respectively.

Thus $(S_\en-z)y_\en=f-q_\en$, and consequently $y_\en=(S_\en-z)^{-1}(f-q_\en)$, where
\begin{multline}
  q_\en=r''_\en+z r_\en+z y_\en\chi_\nu\\
  -\beta\Psi(\nu^{-1}\,\cdot\,)\bigl(\beta y(-0) h_\en(\nu^{-1}\,\cdot\,)+
    \eta^{-1} g_\en(\eps^{-1}\,\cdot\,)+\eps\eta^{-1} v_\en(\eps^{-1}\,\cdot\,)\bigr).
\end{multline}
Recall that $\chi_\nu$ is the characteristic function of $[-\nu,\nu]$.
Owing to Lemmas~\ref{LemmaHenProperties}--\ref{LemmaVenProperties}, we have
\begin{align}\nonumber
    &|y(-0)|\,\left|\Psi\left(\tfrac x\nu\right) h_\en\left(\tfrac x\nu\right)\right|\leq c_1 \|h_\en\|_{C(\cI)}\|f\|\,\chi_\nu(x)\leq c_2\|f\|\,\chi_\nu(x),\\
    &\begin{aligned}\label{EtaPsiGenEst}
    \eta^{-1} \left|\Psi\left(\tfrac x\nu\right)g_\en\left(\tfrac x\eps\right)\right|\leq
    c_3 \eta^{-1}&\chi_\nu(x)\max_{x\in [-\nu,\nu]}|g_\en\left(\tfrac x\eps\right)|\\
    &\leq c_4\eta^{-1} (1+\eta) \|f\|\,\chi_\nu(x)\leq c_5\|f\|\,\chi_\nu(x),
    \end{aligned}\\\label{EtaEpsPsiVenEst}
      &\eps\eta^{-1} |\Psi\left(\tfrac x\nu\right)v_\en\left(\tfrac x\eps\right)|\leq
      c_6 \eps\eta^{-1}\chi_\nu(x)\max_{x\in [-\nu,\nu]}|v_\en\left(\tfrac x\eps\right)|\leq c_7\nu^{1/2}\|f\|\,\chi_\nu(x),
\end{align}
and hence $\|q_\en\|\leq c\varrho(\nu,\eta)\|f\|$, in view of Corollary~\ref{CorWenSmall} and Lemma~\ref{LemmaJumpsAreSmall}. Note also that $\|\chi_\nu\|=(2\nu)^{1/2}$.
Therefore
\begin{multline}\label{Ren-Yen}
     \|(S_\en-z)^{-1}f-y_\en\|=\|(S_\en-z)^{-1}q_\en\|
     \\
     \leq\|(S_\en-z)^{-1}\|\,\|q_\en\|\leq C\varrho(\nu,\eta)\|f\|.
\end{multline}
Note that the resolvents $(S_\en-z)^{-1}$ are uniformly bounded with respect to $\eps$ and $\nu$, because the operators $S_\en$ are self-adjoint.

We next observe that $y_\en-y=r_\en+(w_\en-y)\chi_\nu$.
Thus
\begin{equation}\label{Yen-Y}
\|y_\en-y\|\leq c \varrho(\nu,\eta)\|f\|,
\end{equation}
in view of  Corollary \ref{CorWenSmall} and Lemma \ref{LemmaJumpsAreSmall}. Form this   we deduce for  $z\in \mathbb{C}\setminus\Real$ that
\begin{align*}
     \|(S_\en-z)^{-1}f-(S-z)^{-1}f\|&\leq
     \|(S_\en-z)^{-1}f-y_\en\|+\|y_\en-(S-z)^{-1}f\|
     \\
     &\leq\|(S_\en-z)^{-1}f-y_\en\|+\|y_\en-y\|\leq C\varrho(\nu,\eta)\|f\|,
\end{align*}
for all $f\in L_2(\Real)$, by \eqref{Ren-Yen} and \eqref{Yen-Y}.
The proof is completed by noting that $\varrho(\nu,\eta)$ tends to zero as $\nu\to 0$ and $\eta\to \infty$, that is to say, as $\nu\to 0$ and $\eps\to 0$.
\end{proof}

\subsection{Non-resonant case}
 Here we  prove  the  following  theorem:
\begin{thm}\label{ThmCaseEpsNu-1Go0NR}
    Suppose the potential $\alpha \Phi$ is not resonant; then the operators $S_\en$ converge to the direct sum  $S_-\oplus S_+$  of the Dirichlet half-line Schr\"odinger operators as $\nu\to 0$ and $\eta\to \infty$  in the norm resolvent sense.
\end{thm}

As a matter of fact, this result is implicitly contained in the previous proof.
In the non-resonant case, equation \eqref{EqUalpha} admits only one $L_\infty(\Real)$-solution which is trivial.
Additionally, for each $f\in L_2(\Real)$, the function $y=(S_-\oplus S_+-z)^{-1}f$ satisfies the condition $y(0)=0$.
Roughly speaking, the proof of Theorem~\ref{ThmCaseEpsNu-1Go0NR} can be derived from the previous one
with $u_\alpha$ and $h_\en$ replacing  the zero functions and $y(\pm 0)$ replacing $0$ in the corresponding formulas.

\begin{proof}
In this case the approximation $y_\en$ is rather simpler than \eqref{hatYen}.
Whereas $y(0)=0$, we set
\begin{equation*}
    y_\en(x)=
    \begin{cases}
        y(x)+r_\en(x) &\text{if }|x|>\nu,\\
        \eps g( x/\eps)+\eps^2 v_\en(x/\eps) &\text{if }|x|\leq\nu.
    \end{cases}
\end{equation*}
Here $y=(S_-\oplus S_+-z)^{-1}f$. As above, \, $r_\en$ is a $W_2^2$-corrector of the form \eqref{CorrectoR} and
$v_\en$ is a  solutions of \eqref{ProblVen}.
The function $g$ is a  solutions to the boundary value problem
\begin{equation*}%\label{ProblGenNR}
     g''-\alpha \Phi(t)g= 0,\quad t\in\Real,\qquad
      g'(-1)=y'(-0), \quad g'(1)=y'(+0).
\end{equation*}
Such a solution exists, since $\alpha$ is not an eigenvalue of \eqref{NeumanProblemWithAlpha}.
In addition, $g$ is  linear  outside $\cI$, so it satisfies
the inequalities of the form \eqref{GenEstT}, \eqref{GenPrimeEstR} and \eqref{EtaPsiGenEst}.

Reasoning as in the proof of  Lemma~\ref{LemmaJumpsAreSmall} we deduce that
\begin{align*}
    &|y(\pm\nu)-\eps g(\pm \eta)|\leq |y(\pm\nu)|+\eps |g(\pm \eta)|+\eps^2 |v_\en(\pm\eta)|\leq c_1\nu\|f\|,\\
    &|y'(\pm\nu)- g'(\pm \eta)|\leq |y'(\pm\nu)-y'(\pm 0)|+\eps |v'_\en(\pm\eta)|\leq c_2\nu^{1/2}\|f\|,
\end{align*}
provided $\eta \gg 1$, and  hence that
\begin{equation}\label{RenEstNR}
    \max_{x\in \Real\setminus\{-\nu,\nu\}}\bigl|r^{(k)}_\en(x)\bigr|\leq C\nu^{1/2}\|f\|,\qquad k=0,1,2,
\end{equation}
 by Proposition~\ref{PropW22Corrector}. Furthermore $(S_\en-z)y_\en= f-q_\en$ with
\begin{equation*}
    q_\en(x)=r''_\en(x)+z r_\en(x)+\eps z\chi_\nu(x)g(\tfrac{x}{\eps})
    -\beta\Psi(\tfrac{x}{\nu})\left(\eta^{-1} g(\tfrac{x}{\eps})+\eta^{-1}\eps v_\en(\eps^{-1}\,\cdot\,)\right),
\end{equation*}
by calculations as in the proof of Theorem~\ref{ThmCaseEpsNu-1Go0}. Also $\|q_\en\|\leq c_3\nu^{1/2}\|f\|$,
in view of \eqref{EtaPsiGenEst}, \eqref{EtaEpsPsiVenEst}, and \eqref{RenEstNR}.  This implies
$\|(S_\en-z)^{-1}f-y_\en\|\leq c_4\nu^{1/2}\|f\|$.
The norm resolvent convergence of $S_\en$ towards $S_-\oplus S_+$  now follows
precisely as in the proof of Theorem~\ref{ThmCaseEpsNu-1Go0}.
\end{proof}

%%%%%%%%%%%%%%%%%%%%%%%%%%%%%%%%%%%%%%%%%%%%%%%%%%%%%%%%%%%%%%%%%%%%%%%%%%%%%%%
% Case \eps\nu^{-1}\to const
%%%%%%%%%%%%%%%%%%%%%%%%%%%%%%%%%%%%%%%%%%%%%%%%%%%%%%%%%%%%%%%%%%%%%%%%%%%%%%%
\section{Convergence of the operators $S_\en$. The case $\nu\sim c \eps$.}\label{SecFinite}
In this short section we apply the results of our recent work~\cite{Golovaty:2012} to the case
$\nu\eps^{-1}\to \lambda$ and $\lambda>0$.
The parameters $\eps$ and $\nu$ are in this case connected by the asymptotic relation
$\nu_\eps=\lambda\eps+o(\eps)$ as $\eps\to 0$. Let us consider the operator family
\begin{equation}\label{OprHlambda}
  H_\lambda=
    \begin{cases}
        S(\theta_\alpha,\beta \kappa(\alpha,\lambda))& \text{if } \alpha\in\Lambda_\Phi,\\
        S_-\oplus S_+& \text{otherwise}
    \end{cases}
\end{equation}
for $\lambda>0$, where $\kappa$ is given by \eqref{MapKappa}.
For convenience, we shall write $S_\en(\Phi,\Psi)$ for $S_\en$, and $\kappa(\alpha,\lambda; \Phi,\Psi)$ for $\kappa(\alpha,\lambda)$ indicating the dependence of $S_\en$ and $\kappa$ on  potentials $\Phi$ and $\Psi$.

For the case $\nu=\eps$, it was proved in \cite{Golovaty:2012} that operators $S_{\eps\eps}(\Phi,\Psi)$ converge to $H_1$ in the norm resolvent sense,
as $\eps\to 0$. Moreover, this result is stable under a small perturbation the potential $\Psi$.
If a sequence of  potentials
$\Psi_\eps$ of compact support is uniformly bounded in $L_\infty(\Real)$ and $\Psi_\eps\to \Psi$ in $L_1(\Real)$
as $\eps\to 0$, then $S_{\eps\eps}(\Phi,\Psi_\eps)\to H_1$ in
the sense of the norm resolvent convergence.
Note that all estimates containing $\Psi$ in the proofs of Theorems~4.1 and 5.1 in \cite{Golovaty:2012} remain true with $\Psi$ replaced by $\Psi_\eps$ due to the uniform boundedness of $\Psi_\eps$ in $L_\infty(\Real)$. Next,
the $L_1$-convergence of $\Psi_\eps$ implies
$\kappa(\alpha,1; \Phi,\Psi_\eps)\to \kappa(\alpha,1; \Phi,\Psi)$, as $\eps\to 0$, for all $\alpha\in \Lambda_\Phi$.
Observe also that
\begin{equation*}
    S_\el(\Phi,\Psi)=-\frac{d^2}{dx^2}+\frac{\alpha}{\eps^2}\Phi\left(\frac{x}{\eps}\right)
    +\frac{\beta}{\lambda\eps}\Psi\left(\frac{x}{\lambda\eps}\right) =S_{\eps, \eps}(\Phi,\Upsilon)
\end{equation*}
with $\Upsilon=\frac1\lambda\Psi(\frac 1\lambda\,\cdot\,)$.
Next, we see that
\begin{multline*}
  \kappa(\alpha,1; \Phi,\Upsilon)=\theta_\alpha^{-1} \int_\Real\frac{1}{\lambda}\Psi\left(\frac{t}{\lambda}\right) u^2_\alpha( t)\,dt\\
  =
    \theta_\alpha^{-1}\int_\Real\Psi\left(\tau\right) u^2_\alpha(\lambda \tau)\,d\tau=\kappa(\alpha,\lambda; \Phi,\Psi).
\end{multline*}
Therefore $S_\el(\Phi,\Psi)\to H_\lambda$ as $\eps\to 0$ in the sense of uniform convergence of  resolvents.

Repeating the previous scaling arguments leads to $S_\en(\Phi,\Psi)=S_\el(\Phi,\Psi_\eps)$,
where $\Psi_\eps=\gamma_\eps\Psi(\gamma_\eps\,\cdot\,)$ and  $\gamma_\eps=\lambda\eps/\nu_\eps$.
Since $\gamma_\eps\to 1$ as $\eps$ goes to $0$, $\Psi_\eps\to \Psi$ in $L_1(\Real)$ as $\eps\to 0$. Hence both operators $S_\en(\Phi,\Psi)$ and $S_\el(\Phi,\Psi)$
converge to the same limit $H_\lambda$.
We have proved:

\begin{thm}
If the ratio $\nu/\eps$ tends to a finite positive number $\lambda$ as $\nu, \eps\to 0$,
then  $S_\en$  converge to the operator $H_\lambda$ defined by \eqref{OprHlambda} in the norm resolvent sense.
\end{thm}

%%%%%%%%%%%%%%%%%%%%%%%%%%%%%%%%%%%%%%%%%%%%%%%%%%%%%%%%%%%%%%%%%%%%%%%%%%%%%%%
% Case \eps\nu^{-1}\to \infty
%%%%%%%%%%%%%%%%%%%%%%%%%%%%%%%%%%%%%%%%%%%%%%%%%%%%%%%%%%%%%%%%%%%%%%%%%%%%%%%

\section{Convergence of the operators $S_\en$. The case $\nu\eps^{-1}\to 0$.}\label{SecInfty}
We discuss in this section the case of the fast contracting $\Psi$-shaped potential  relative to the $\Phi$-shaped one. Therefore that $\nu\eps^{-1}\to 0$ as $\nu, \eps\to 0$.  First we note that if $\eps\to 0$ and $\eta\to 0$, then $\nu\to 0$.
As in Section~\ref{SecZero}, the resonant and non-resonant cases will be treated separately.

\subsection{Resonant case}
Let us consider  the operator $S(\theta_\alpha, \beta\mu_\alpha)$, where $\mu_\alpha=\mu(\alpha)$ and the mapping $\mu\colon \Lambda_\Phi\to \Real$ is given by \eqref{MapMu}.

\begin{thm}\label{ThmCaseEpsNu-1GoInfty}
   Suppose   $\Phi, \Psi \in \mathcal{P}$ and $\alpha\in\Lambda_\Phi$; then the operator family $S_\en$  converges to   $S(\theta_\alpha, \beta\mu_\alpha)$ in the norm resolvent sense, as $\eps,\eta \to 0$.
\end{thm}

Given $f\in L_2(\Real)$ and $z\in \mathbb{C}\setminus\Real$,  we write $y=(S-z)^{-1}f$, where $S=S(\theta_\alpha, \beta\mu_\alpha)$.
Note that $y$ satisfies the conditions
\begin{equation}\label{YcondsInfty}
    y(+0)=\theta_\alpha y(-0), \quad y'(+0)=\theta_\alpha^{-1} y'(-0)+\beta\mu_\alpha y(-0).
\end{equation}
Let us next guess $y_\en$
has the form
\begin{equation}\label{ApproxWeInfty}
  y_\en(x)=
  \begin{cases}
    y(x)+r_\en(x)&\text{for }|x|>\eps,\\
    y(-0)u_\alpha(x/\eps)+\eps g_\en( x/\eps)+\beta\nu\eps h_\en(x/\nu)+\eps^2 v_\en(x/\eps)&\text{for }|x|\leq\eps,
  \end{cases}
\end{equation}
 where $g_\en$, $h_\en$, and $v_\en$ are  solutions to the Cauchy problems
\begin{align}
\label{ProblGenInfty}
    &\begin{cases}\displaystyle
     g''-\alpha \Phi(t)g=\beta y(-0)\,\eta^{-1}\Psi(\eta^{-1} t)u_\alpha(t),\qquad t\in\Real,\\\displaystyle
      g(-1)=0, \quad g'(-1)=y'(-0);
    \end{cases}\\\label{ProblHenInfty}
     &\hskip12pt h''=  \Psi(t)g_\en(\eta t),\quad t\in\Real,\qquad
      h(-1)=0,\quad h'(-1)=0;
      \\ \label{ProblVenInfty}
     &\begin{cases}
        -v''+\alpha\Phi(t)v+\beta\eps\eta^{-1}\, \Psi(\eta^{-1} t)v=f(\eps t),\quad t\in \Real,\\
       \phantom{-} v(-1)=0,\quad v'(-1)=0
      \end{cases}
\end{align}
respectively. As above, $u_\alpha$ is the half-bound state for the potential $\alpha\Phi$, and
$r_\en$  adjusts this approximation so as to obtain an ele\-ment of $\dmn S_\en$.
According to Proposition~\ref{PropW22Corrector}, there exists  a  corrector function $r_\en$ that vanishes in $(-\eps,\eps)$.

\begin{lem}\label{LemmaGenPropertiesInfty}
If the ratio of $\nu$ to $\eps$ remains bounded as $\nu, \eps\to 0$, then
there exists a constant $C$ such that for all $f\in L_2(\Real)$
 \begin{equation}\label{GenEstInfty}
    \|g_\en\|_{C(\cI)}\leq C\|f\|.
 \end{equation}
 In addition,  $g'_\en(1)=y'(+0)+ O(\eta)\|f\|$ as  $\eps,\eta\to 0$.
\end{lem}
\begin{proof}
Our proof starts with the observation that the right-hand side of equation \eqref{ProblGenInfty} contains a $\delta$-like sequence, namely
\begin{equation}\label{DeltaLikeSeqInfty}
    \eta^{-1}\Psi(\eta^{-1} t)\to \left(\int_\Real \Psi\,dt\right)\delta(x)\quad \text{in } W_2^{-1}(\cI)
\end{equation}
as $\eta\to0$.
Let $g_0$ be the solution of \eqref{EqUalpha} obeying the initial conditions $g_0 (-1)=0$ and $g_0 '(-1)=1$.
Then  $g_\en$ can be represented as $g_\en=y'(-0)g_0 +\beta y(-0)\hat{g}_\en$,
where $\hat{g}_\en$ solves  the equation $g''-\alpha \Phi g=\eta^{-1}\Psi(\eta^{-1} \cdot\,)u_\alpha$ and satisfies  zero initial conditions at $t=-1$. Next, $\hat{g}_\en$ converges in $W_2^1(\cI)$ to the solution $\hat{g}$ of the problem
\begin{equation*}
    g''-\alpha \Phi(t)g= u_\alpha(0)\,\left(\int_\Real \Psi\,dt\right)\delta(x),\quad t\in\cI,\qquad
    g(-1)=0, \quad g'(-1)=0,
\end{equation*}
which is clear from the explicit representation of $\hat{g}_\en$ of the form \eqref{CPSolRepresentation}. Thus the convergence in $W_2^{1}(\cI)$ implies the uniform convergence of $\hat{g}_\en$ to $\hat{g}$ in  $\cI$, and consequently $\hat{g}_\en$ is  uniformly bounded  in $\eps$ and $\nu$ provided $\eta<c$.
From this we see that $ \|g_\en\|_{C(\cI)}\leq |y'(-0)|\,\|g_0 \|_{C(\cI)}+|\beta|\,|y(-0)|\,\|\hat{g}_\en\|_{C(\cI)}\leq C\|f\|$, by \eqref{EstY(pm0)}.

Multiplying equation \eqref{ProblGenInfty} by $u_\alpha$ and integrating on $\cI$ by parts  yield
\begin{equation*}
   \theta_\alpha g'_\en(1)- y'(-0)=
 \beta y(-0)\eta^{-1}\int_{-1}^1\Psi(\eta^{-1} s)u^2_\alpha(s)\,ds.
\end{equation*}
Since $u_\alpha(t)=u_\alpha(0)+O(t)$ as $t\to 0$, we have
\begin{multline*}
   g'_\en(1)=\theta_\alpha^{-1} \left(y'(-0)+ \beta y(-0)u^2_\alpha(0) \int_\Real\Psi\,ds\right)+ O(\eta)\|f\|\\
   =\theta_\alpha^{-1} y'(-0)+ \beta\mu_\alpha y(-0)+ O(\eta)\|f\|,\quad \eta\to 0,
\end{multline*}
by \eqref{DeltaLikeSeqInfty} and \eqref{MapMu}. Therefore the asymptotic relation for $g'_\en(1)$ follows from  \eqref{YcondsInfty}.
\end{proof}

\begin{lem}\label{LemmaHenPropertiesInfty}
There exist constants $C_1$ and $C_2$, independent of $f$, such that
\begin{align}\label{HenEstTInfty}
    &|h_\en(t)|\leq C_1(1+|t|)\|f\|, &&t\in\Real,\\\label{HenPrimeEstInfty}
    &|h'_\en(t)|\leq C_2\|f\|,  &&t\in\Real
\end{align}
for all  $\eps$ and $\nu$ whenever the ratio of $\nu$ to $\eps$ is small enough.
\end{lem}
\begin{proof}
As in the proof of Lemma~\ref{LemmaGenProperties}, equation \eqref{ProblHenInfty} gives
\begin{equation*}
     h_\en(t)=t\int_{-1}^1\Psi(s)g_\en(\eta s)\,ds-\int_{-1}^1s\Psi(s)g_\en(\eta s)\,ds\quad \text{for } t\geq1
\end{equation*}
and $h_\en(t)=0$ for $t\leq-1$.
If $|\eta|\leq 1$, then  \eqref{HenEstTInfty},  \eqref{HenPrimeEstInfty} follow   from \eqref{GenEstInfty}.
\end{proof}

\begin{lem}\label{LemmaVenPropertiesInfty}
There exist constants $C$ independent of $f$ such that
\begin{equation}\label{VenEstInfty}
   \|v_\en\|_{C^1(\cI)}\leq C\eps^{-1/2}\|f\|
\end{equation}
for all $\eps$ and $\nu$ small enough.
\end{lem}
\begin{proof}
  Let $v_\eps$ be a solution of the auxiliary Cauchy problem
  \begin{equation*}
     -v''_\eps+\alpha\Phi(t)v_\eps=f(\eps t),\quad t\in \Real,\quad
     v_\eps(-1)=0,\quad v'_\eps(-1)=0.
  \end{equation*}
In view of Proposition~\ref{PropCauhyProblEst} we have
  \begin{equation*}
    v_\eps(t)=\int_{-1}^t k(t,s) f(\eps s)\,ds,
  \end{equation*}
where $k=k(t,s)$  is a continuously differentiable function on $\Real^2$.
Therefore
  \begin{equation}\label{AuxVepsEsrInfty}
    \|v_\eps\|_{C^1(\cI)}\leq c_1 \|k\|_{C^1(\cI\times\cI)}\int_{-1}^1|f(\eps s)|\,ds
    \leq c_2\eps^{-1} \int_{-\eps}^\eps|f(\tau)|\,d\tau
    \leq c_3\eps^{-1/2}\|f\|.
  \end{equation}
Next, the function $\vartheta_\en=v_\en-v_\eps$ solves the problem
 \begin{equation*}
     -\vartheta''_\eps+\alpha\Phi(t)\vartheta_\eps=-\beta\eps\eta^{-1}\, \Psi(\eta^{-1} t)v_\en,\quad t\in \Real,\quad
     \vartheta_\eps(-1)=0,\quad \vartheta'_\eps(-1)=0.
  \end{equation*}
We conclude from this that
\begin{multline*}
   \|\vartheta_\en\|_{C^1(\cI)}\leq c_4 \eps\eta^{-1}\|k\|_{C^1(\cI\times\cI)}\int_{-1}^1 |\Psi(\eta^{-1} s)||v_\en(s)|\,ds\\
   \leq c_5 \eps\|v_\en\|_{C^1(\cI)}\,\eta^{-1}\int_{-1}^1 |\Psi(\eta^{-1} s)|\,ds
   \\
   \leq
   c_5 \eps\|v_\en\|_{C^1(\cI)}\int_\Real |\Psi(\tau)|\,d\tau\leq c_6\eps\|v_\en\|_{C^1(\cI)}.
\end{multline*}
Hence,  $\|v_\en-v_\eps\|_{C^1(\cI)}\leq c_6\eps\|v_\en\|_{C^1(\cI)}$, and consequently
$$(1-c_6\eps)\|v_\en\|_{C^1(\cI)}\leq \|v_\eps\|_{C^1(\cI)}.$$
That $ \|v_\en\|_{C^1(\cI)}\leq C\eps^{-1/2}\|f\|$ follows from estimate \eqref{AuxVepsEsrInfty} for $\eps$ small enough.
\end{proof}

Lemmas \ref{LemmaGenPropertiesInfty}--\ref{LemmaVenPropertiesInfty} have the following corollary.
\begin{cor}\label{CorWenSmallInfty}
The function $y_\en$ is bounded in $[-\eps,\eps]$ uniformly in $\eps$ and $\nu$ provided  $\nu/\eps\leq 1$, and
   $\max_{|x|\leq \eps} |y_\en(x)|\leq C\|f\|$ with some constant $C$ being independent of $f$.
\end{cor}

The function $w_\en=y_\en-r_\en$ and its first derivative have  the jumps at $x=\pm \eps$:
\begin{align*}
    &[w_\en]_{-\eps}=y(-0)-y(-\eps),\qquad     [w'_\en]_{-\eps}=y'(-0)-y'(-\eps),\\
    &[w_\en]_{\eps}=y(\eps)-\theta_\alpha y(-0)-\eps g_\en(1)-\beta\nu\eps\, h_\en(\eta^{-1})-\eps^2 v_\en(1),\\
    &[w'_\en]_{\eps}=y'(\eps)-g'_\en(1)- \eps  (\beta\,h'_\en(\eta^{-1})+ v'_\en(1)).
\end{align*}
In view of \eqref{EstY(t)-Y(0)}, \eqref{GenEstInfty}, \eqref{HenEstTInfty}, \eqref{VenEstInfty}, and \eqref{YcondsInfty}, we conclude  that  three of the jumps can be bounded by $c_1\eps^{1/2}\|f\|$. As for the last one, we have
\begin{equation*}
    \left|[w'_\en]_{\eps}\right|\leq |y'(\eps)-y'(+0)|+ c_2\eta\|f\|+c_3\eps (|h'_\en(\eta)|+ |v'_\en(1)|)
    \leq c_2(\eps^{1/2}+\eta)\|f\|,
\end{equation*}
by  \eqref{HenPrimeEstInfty}, \eqref{VenEstInfty},  and Lemma \ref{LemmaGenPropertiesInfty}. We can now repeatedly apply  Proposition~\ref{PropW22Corrector} to deduce
\begin{equation}\label{RenEstInfty}
    \max_{x\in \Real\setminus\{-\eps,\eps\}}\bigl|r^{(k)}_\en(x)\bigr|\leq C\sigma(\eps,\eta)\|f\|
\end{equation}
for $k=0,1,2$, where $\sigma(\eps,\eta)=\eps^{1/2}+\eta$.

\begin{proof}[Proof of Theorem~\ref{ThmCaseEpsNu-1GoInfty}]
As in the proof of Theorem~\ref{ThmCaseEpsNu-1Go0} we introduce the notation $f_\en=(S_\en-z)y_\en$.
It is easy to check that $f_\en(x)=f(x)-r''_\en(x)-z r_\en(x)$ for $|x|>\eps$.  Next, for $|x|<\eps$, we have
\begin{align*}\textstyle
  f_\en(x)=&\left(-\tfrac{d^2}{dx^2}+\alpha\eps^{-2}\Phi\left(\tfrac{x}{\eps}\right)
    +\beta\nu^{-1}\Psi\left(\tfrac{x}{\nu}\right)-z\right)y_\en(x)
\\
     = &\eps^{-2} y(-0)\Bigl\{-u''_\alpha\left(\tfrac{x}{\eps}\right)+\alpha \Phi\left(\tfrac x\eps\right)u_\alpha\left(\tfrac{x}{\eps}\right)\Bigr\}
\\
    +&\eps^{-1} \Bigl\{ -g_\en''\left(\tfrac{x}{\eps}\right)+\alpha \Phi\left(\tfrac x\eps\right)g_\en\left(\tfrac{x}{\eps}\right)+
    \beta \eta^{-1}y(-0)\Psi\left(\tfrac x\nu\right) u_\alpha\left(\tfrac x\eps\right)\Bigr\}
\\
    +&\beta\eta^{-1}  \Bigl\{-h''_\en\left(\tfrac x\nu\right)+ \Psi\left(\tfrac x\nu\right) g_\en\left(\tfrac x\eps\right)\Bigl\}
\\
       +&\Bigl\{-v''_\en\left(\tfrac{x}{\eps}\right)+\alpha\Phi\left(\tfrac x\eps\right)v_\en\left(\tfrac{x}{\eps}\right)+\beta\eps^2\nu^{-1}\, \Psi\left(\tfrac x\nu\right)v_\en\left(\tfrac{x}{\eps}\right)\Bigr\}
\\
    +&\alpha\beta \eta\, \Phi\left(\tfrac x\eps\right)h_\en\left(\tfrac x\nu\right)
    +\beta^2\eps \,\Psi\left(\tfrac x\nu\right)h_\en\left(\tfrac x\nu\right)- z y_\en(x)
   \\ =&f(x)+\Bigl\{\alpha \eta\, \Phi\left(\tfrac x\eps\right)
   +\beta\eps\, \Psi\left(\tfrac x\nu\right)\Bigr\}\beta h_\en\left(\tfrac x\nu\right)
   - z y_\en(x),
\end{align*}
since  $u_\alpha$, $g_\en$, $h_\en$, and $v_\en$ are solutions to equations \eqref{EqUalpha} and \eqref{ProblGenInfty}--\eqref{ProblVenInfty} respectively.
Then $f_\en=f-q_\en$, where
$$
    q_\en=r''_\en+z r_\en+z y_\en\chi_\eps-\left(\alpha \eta \Phi(\eps^{-1}\,\cdot\,)
   +\beta\eps \Psi(\nu^{-1}\,\cdot\,)\right)\beta h_\en(\nu^{-1}\,\cdot\,).
$$
As above,  $\chi_\eps$ is the characteristic function of $[-\eps,\eps]$.
Consequently, we conclude from Lemma~\ref{LemmaHenPropertiesInfty}  that
\begin{align*}
&\begin{aligned}
\eta \left|\Phi\left(\tfrac x\eps\right)h_\en\left(\tfrac x\nu\right)\right|\leq
c_1\eta\chi_\eps(x)\max_{|x|\leq\eps}&|h_\en\left(\tfrac x\nu\right)|\\
&\leq c_2\eta (1+\eta^{-1})\|f\|\,\chi_\eps(x)\leq c_3\|f\|\,\chi_\eps(x),
\end{aligned}
    \\
&\eps \left|\Psi\left(\tfrac x\nu\right) h_\en\left(\tfrac x\nu\right)\right|\leq
      c_4\eps\chi_\nu(x)\max_{|x|\leq\nu}|h_\en\left(\tfrac x\nu\right)|\leq c_5\eps\|f\|\,\chi_\nu(x),
\end{align*}
hence that $\|q_\en\|\leq c\sigma(\eps,\eta)\|f\|$, in view of Corollary~\ref{CorWenSmallInfty} and estimate \eqref{RenEstInfty}.
Thus $y_\en=(S_\en-z)^{-1}f+(S_\en-z)^{-1}q_\en$, and therefore
\begin{equation*}
     \|(S_\en-z)^{-1}f-y_\en\|\leq\|(S_\en-z)^{-1}\|\|q_\en\|\leq c_6\sigma(\eps,\eta)\|f\|.
\end{equation*}
By arguments that are completely analogous to those presented in the proof of Theorem~\ref{ThmCaseEpsNu-1Go0} we
conclude that $\|(S(\theta_\alpha, \beta\mu_\alpha)-z)^{-1}f-y_\en\|\leq C\sigma(\eps,\eta)\|f\|$, and finally that operators $S_\en$ converge to $S(\theta_\alpha, \beta\mu_\alpha)$ in the norm resolvent sense as $\eps$ and $\eta$ tend to zero.
\end{proof}

\subsection{Non-resonant case}
 Assume  $\alpha$ does not belongs to the resonant set $\Lambda_\Phi$, and write $y=(S_-\oplus S_+-z)^{-1}f$ for $f\in L_2(\Real)$.

\begin{thm}\label{ThmCaseEpsNu-1GoInftyNR}
    If $\alpha\not\in\Lambda_\Phi$,  then the operator family $S_\en$ defined by \eqref{Sen} converges to the direct sum  $S_-\oplus S_+$  in the norm resolvent sense as $\eps, \eta\to 0$.
\end{thm}

\begin{proof} In this case the approximation $y_\en$ may be greatly simplified,  since $y(0)=0$.
Looking at  asymptotics \eqref{ApproxWeInfty}, we set
\begin{equation*}
y_\en(x)=
\begin{cases}
  y(x)+r_\en(x)\quad &\text{for }|x|>\eps,\\
  \eps g( x/\eps)+\beta\nu\eps\, h_\en(x/\nu)+\eps^2 v_\en(x/\eps)\quad &\text{for }|x|\leq\eps,
\end{cases}
\end{equation*}
where $g$ and $h_\en$ are  solutions to the problems
\begin{align*}
    & g''-\alpha \Phi(t)g=0,\quad t\in\Real, && g'(-1)=y'(-0), \quad g'(1)=y'(0);
\\
    & h''= \Psi(t)g(\eta t),\quad t\in\Real, &&     h(-1)=0,\quad h'(-1)=0
\end{align*}
respectively. As above, $v_\en$ is a solution of \eqref{ProblVenInfty}, and the corrector function $r_\en$ is of the form \eqref{CorrectoR} and provides the inclusion
$y_\en\in W_2^2(\Real)$.
The rest of the proof is similar to the proof of Theorem~\ref{ThmCaseEpsNu-1GoInfty}.
\end{proof}

\emph{Acknowledgements.} I would like to thank Rostyslav Hryniv and Alexander Zolotaryuk for stimulating
discussions.


\begin{thebibliography}{10}

\bibitem{AktosunKlaus:2001}
    T. Aktosun and M. Klaus, Small-energy asymptotics for the Schr\"{o}dinger equation on the line. Inverse Problems (4) \textbf{17} (2001), 619--632.

\bibitem{AlbeverioCacciapuotiFinco:2007}
    S. Albeverio, C. Cacciapuoti, and D. Finco,
    Coupling in the singular limit of thin quantum waveguides.
    J. Math. Phys.  (3) {\bf 48} (2007), 032103, 21pp.

\bibitem{AlbeverioGesztesyHoeghKrohn:1982}
    S.~Albeverio, F.~Gesztesy,  and R.~H{\o}egh-Krohn,
    The low energy expansion in nonrelativistic scattering theory.
    {Annales de l'I. H. P., section A}, (1) {\bf 37} (1982), 1--28.


\bibitem{AlbeverioGesztesyHoeghKrohnHolden2edition}
    S.~Albeverio, F.~Gesztesy,  R.~H{\o}egh-Krohn, and H. Holden,
   \textit{Solvable Models in Quantum Mechanics, 2nd revised ed.}
    AMS Chelsea Publ. 2005.


\bibitem{AlbeverioGesztesyHoeghKrohnStreit:1983}
    S.~Albeverio, F.~Gesztesy, R.~H{\o}egh-Krohn, and L.~Streit,
    Charged particles with short range interactions.
    Annales de l'I. H. P., sect. A. (3) {\bf 38} (1983), 263--293.


\bibitem{AlbeverioHoeghKrohn:1981}
    S.~Albeverio and R.~H{\o}egh-Krohn,
    Point interactions as limits of short range interactions.
    {J. Operator Theory} {\bf 6} (1981), 313--339.


\bibitem{AlbeverioKurasov}
    S.~Albeverio and P. Kurasov,
     \textit{Singular Perturbations of Differential Operators and Schr\"{o}-dinger Type Operators.}
    Cambridge Univ. Press 2000.




\bibitem{BolleGesztesyDanneels:1988}
    D.~Boll\'{e}, F.~Gesztesy, and C.~Danneels,
    Threshold scattering in two dimensions.
    {Annales de l'I. H. P., section A}, (2) {\bf 48} (1988), 175--204.

\bibitem{BolleGesztesyKlaus:1987}
    D.~Boll\'{e}, F.~Gesztesy, and M.~Klaus,
    Scattering theory for one-dimensional systems with $\int dx V(x) = 0$.
    {J. Math. Anal. Appl.} {\bf 122} (1987), 496--518.

\bibitem{BolleGesztesyWilk:1985}
    D.~Boll\'{e}, F.~Gesztesy, and S.~F.~J.~Wilk,
    A complete treatment of low-energy scattering in one dimension.
    {J. Operator Theory} {\bf 13} (1985), 3--31.

\bibitem{BrascheFigariTeta}
     J. F. Brasche, R. Figari, and A. Teta,
    Singular Schr\"{o}dinger operators as limits of point interaction Hamiltonians.
    {Potential Anal.} (2) {\bf 8} (1998), 163--178.

\bibitem{BrascheNizhnik:2011}
    J. F. Brasche and  L. Nizhnik,
    One-dimensional Schr\"{o}dinger operators with $\delta'$-interactions on a set of Lebesgue measure zero.
    arXiv:1112.2545 [math.FA]. 22pp.

\bibitem{CacciapuotiExner:2007}
    C.~Cacciapuoti and P.~Exner,
    Nontrivial edge coupling from a Dirichlet network squeezing:
    the case of a bent waveguide.
    {J. Phys. A: Math. Theor.} (26) {\bf 40} (2007), F511--F523.

\bibitem{CacciapuotiFinco:2007}
    C. Cacciapuoti and D. Finco,
    Graph-like models for thin waveguides with Robin boundary conditions. 	
    {Asymptotic Analysis}. (3--4) {\bf 70} (2010), 199--230.

\bibitem{Cacciapuoti:2011}
     C. Cacciapuoti,
    Graph-like asymptotics for the Dirichlet Laplacian in connected tubular domains.
    arXiv:1102.3767v2 [math-ph], 22 pp.

\bibitem{ChristianZolotarIermak03}
     P.~L.~Christiansen,  H.~C.~Arnbak,  A.~V.~Zolotaryuk,  V.~N.~Ermakov, and
     Y.~B.~Gaididei,
     On the existence of resonances in the transmission probability for interactions
     arising from derivatives of Dirac's delta function. J. Phys. A: Math. Gen. {\bf 36} (2003),
    7589--7600.

\bibitem{CurgusLangerJDE:1989}
    B. \'{C}urgus and H. Langer,
    A Krein space approach to symmetric ordinary differential operators with an indefinite weigth function.
    J. Diff. Eq. \textbf{79} (1989), no.~1, 31--61.


\bibitem{DeiftTrubowitz:1979} P.~Deift and E.~Trubowitz,
    Inverse scattering on the line.
    {Comm. Pure Appl. Math.} {\bf 32} (1979), 121--251.

\bibitem{ExnerNeidhardtZagrebnov}
    P.~Exner, H.~Neidhardt, and V.~A.~Zagrebnov,
    Potential approximations to $\delta'$: an inverse Klauder phenomenon with norm-resolvent convergence.
    {Comm. Math. Phys.} (3) {\bf 224} (2001), 593--612.


\bibitem{GolovatyManko:2009} Yu.~Golovaty and  S.~Man'ko,
    Solvable models for the Schr\"odinger operators
    with $\delta'$-like potentials.  { Ukr. Math. Bulletin} (2) {\bf 6} (2009), 169--203;
    arXiv:0909.1034v2 [math.SP].


\bibitem{GolovatyHryniv:2010}
    Yu.~Golovaty and R.~Hryniv,
    On norm resolvent convergence of Schr\"{o}dinger operators with $\delta'$-like potentials.
    {J. Phys. A: Math. Theor.} (15) {\bf 43} (2010), 155204, 14pp;
    Corrigendum  J. Phys. A: Math. Theor. \textbf{44} (2011), 049802;  arXiv:0911.1046[math.SP].

\bibitem{GolovatyHryniv:2011}
    Yu.~Golovaty and R.~Hryniv,
    Norm resolvent convergence of Schr\"{o}dinger operators with singularly scaled potentials.
    arXiv:1108.5345[math.SP]. 30pp.

\bibitem{Golovaty:2012}
    Yu.~Golovaty,
    Schr\"{o}dinger operators with  $(\alpha\delta'+\beta \delta)$-like potentials: norm resolvent convergence and solvable models. arXiv:1201.2610v2 [math.SP], accepted for publication in Methods of Funct. Anal. Topology (2012).


\bibitem{GoriunovMikhailetsMFAT:2010}
    A. Goriunov, V. Mikhailets,
    \textit{Regularization  of  singular  Sturm-Liouville equations},
    Methods Func. Anal. Topol., \textbf{16} (2010), no.~2,  120--130.

\bibitem{GoriunovMikhailetsMN:2010}
    A. Goriunov, V. Mikhailets,
    \textit{Resolvent convergence of Sturm–Liouville operators with singular potentials},
    Mathematical Notes, \textbf{87} (2010), no.~2, 287--292.


\bibitem{IsmagilovKostyuchenko:2010} R. S. Ismagilov and A. G. Kostyuchenko,
    Spectral asymptotics for the Sturm--Liouville operator with point interaction.
    Funct. Anal. Appl. (4) {\bf 44} (2010) 253--258.


\bibitem{Klaus:1982} M. Klaus,
    Some applications of the Birman--Schwinger principle.
    {Helvetica Physica Acta} {\bf 55} (1982), 49--68.

\bibitem{Klaus:1988} M.~Klaus,
    Low-energy behaviour of the scattering matrix for the Schr\"odinger equation on the line.
    {Inverse Problems} (2) {\bf 4} (1988), 505--512.

\bibitem{KlausSimonI:1980}
    M.  Klaus and B.  Simon, Coupling constant thresholds in  nonrelativistic quantum  mechanics.  I.  Short-range two~body case. Annals  Of  Physics  {\bf 130} (1980), 251--281.

\bibitem{KlausSimonII:1980}
    M.  Klaus and B.  Simon, Coupling constant thresholds in nonrelativistic quantum mechanics II. Two  cluster thresholds  in $N$-body systems. Commun.  Math.  Phys. {\bf 78} (1980),  153--168.


\bibitem{KostenkoMalamud:2010}
    A.~Kostenko and M.~Malamud, 1-D Schr\"{o}dinger operators with local point interactions on a discrete set.
    J. Differential Equations {\bf 249} (2010) 253--304.

\bibitem{MankoJPA:2010}
    S.~S. Man'ko, On $\delta'$-like potential scattering on star graphs.
    J. Phys. A: Math. Theor (44) \textbf{43} (2010), 445304 (14pp).

\bibitem{MikhailetsMolyboga:2008}
    V. Mikhailets and V. Molyboga, One-dimensional Schr\"{o}dinger
    operators with singular periodic potentials. Methods Funct. Anal. Topology. (2)
    {\bf 14} (2008),  184--200.

\bibitem{MikhailetsMolyboga:2009}
    V. Mikhailets and V. Molyboga, Spectral gaps of the one-dimensional Schr\"{o}dinger
    operators with singular periodic potentials. Methods Funct. Anal. Topology. (1)
    {\bf 15} (2009),  31--40.

\bibitem{Nizhik:2006FAA}
   L. P. Nizhnik, {A one-dimensional Schr\"odinger operator with point interactions on Sobolev spaces}. Funct. Anal. Appl. (2) {\bf 40} (2006), 143--147.


\bibitem{SebaHalfLine:1985}
    P. \v{S}eba,
    Schr\"{o}dinger particle on a half line.
    {Lett. Math. Phys.} (1) {\bf 10} (1985), 21--27.

\bibitem{SebRMP:1986}
    P.  \v{S}eba, Some remarks on the $\delta'$-interaction in one dimension.
    Rep. Math. Phys. (1) {\bf 24} (1986), 111--120.

\bibitem{ShkalikovSavchukMatNotes1999}
    A. M. Savchuk, A. A. Shkalikov, \textit{Sturm-Liouville operators with singular potentials }// Math. Notes \textbf{66} (1999),  N~6, 741--753.

\bibitem{ShkalikovSavchuk:2003TMMO}
    A. M. Savchuk and A. A. Shkalikov, {Sturm-Liouville operators with distribution potentials}.
    Tr. Mosk. Mat. O.-va  {\bf 64} (2003), 159--212.

\bibitem{Zolotaryuk08}
    A. V. Zolotaryuk,
    {Two-parametric resonant tunneling across the $\delta'(x)$ potential}.
    Adv. Sci. Lett. {\bf 1} (2008), 187--191.

\bibitem{Zolotaryuk09}
    A. V. Zolotaryuk,
    Point interactions of the dipole type defined through a three-parametric power regularization.
    {J. Phys. A: Math. Theor.} {\bf 43} (2010), 105302 (21 pp).

\bibitem{Zolotaryuk10}
    A. V. Zolotaryuk,
    Boundary conditions for the states with resonant tunnelling across the $\delta'$-potential.
    {Physics Letters A} (15--16) {\bf 374} (2010), 1636--1641.
\end{thebibliography}
\end{document}